\documentclass[final,a4paper,12pt]{article}

\usepackage{amsmath,amsthm,amssymb,amsfonts}
\usepackage{indentfirst}
\usepackage[notref,notcite]{showkeys}
\usepackage{epsfig,graphicx,color}

\newcommand{\R}{\mathbb{R}}

\renewcommand{\S}{{\mathbb{S}^{N-1}}}
\newcommand{\un}{\mathbf{1}}
\newcommand{\eps}{\varepsilon}

\newtheorem{defi}{Definition}
\newtheorem{lem}{Lemma}
\newtheorem{prop}{Proposition}
\newtheorem{thm}{Theorem}
\newtheorem{cor}{Corollary}

\theoremstyle{remark}
\newtheorem{rem}{Remark}

\newtheorem{ex}{Example}

\begin{document}

\begin{center}

{\Large \bf \sffamily Level set approach for fractional mean curvature flows}
\bigskip

{\small Cyril Imbert\footnote{Universit\'e Paris-Dauphine, CEREMADE,
    place de Lattre de Tassigny, 75775 Paris cedex 16, France, 
\texttt{imbert@ceremade.dauphine.fr}}}

\bigskip

{\today}

\end{center}

\begin{quote} \footnotesize
\noindent \textsc{Abstract.}
This paper is concerned with the study of a geometric flow whose law involves a 
singular integral operator. This operator is used to define a
non-local mean curvature of a set. Moreover the associated
 flow appears in two important applications: dislocation dynamics and 
phasefield theory for fractional reaction-diffusion equations. 
It is defined by using the level set method. The main results of this paper
are: on one hand, the proper level set formulation of the geometric flow;
on the other hand, stability and comparison results for the geometric equation associated
with the flow.
\end{quote}
\vspace{5mm}

\medskip

\bigskip

\section{Introduction}

In this paper, we define a geometric flow whose law is non-local. We recall that a  
geometric flow of a set $\Omega$  is a family $\{ \Omega_t \}_{t>0}$ such that 
the velocity of a point $x \in \partial \Omega_t$ along its outer normal $n(x)$ is a given
function of $x$ and $n (x)$ for instance. In our case, this velocity does not only depend
on $x$ and $n(x)$ but also on a \emph{fractional mean curvature} at $x$. 
Our motivation comes from two different problems: dislocation dynamics and phasefield theory
for fractional reaction-diffusion equations. 

\subsection{Motivation and existing results}

Mathematical study of non-local moving fronts recently attracted a lot
of attention (see in particular \cite{cardaliaguet} and references therein). 
An important application is the study of  dislocation dynamics \cite{ahlm}. 

\subsubsection*{Dislocation dynamics}

Dislocations are linear defects 
in crystals and the study of their motion gives rise to the study of a non-local geometric 
flow. In recent years, several papers were dedicated to this problem. We next briefly recall 
 the results contained in these papers.

A dislocation creates an elastic field  in the whole space $\R^3$
and this field creates a force (called the
Peach-Koehler force) that acts not only on the dislocation that created
it (self-force) but also on all dislocations in the material. 
We restrict ourselves here to the case of a single curve. We also assume that 
this curve moves in a plane (called the slip plane). 

The level set approach \cite{oshersethian,cgg,es} is a general method for constructing
moving interfaces. It consists in representing  $\Omega_t$ as  zero level 
sets of functions $u (t,\cdot)$. The geometric law satisfied by the interface $\partial \Omega_t$ 
is thus translated into an evolution equation satisfied by $u$. 
This approach is used in \cite{ahlm} to describe the dynamics
of a dislocation line. If $\partial \Omega_t$ is the zero level set
of a function $u (t,\cdot)$,  the following non-local eikonal equation is obtained
$$
\partial_t u = (c_1 (x) + \kappa [x,u] )  |Du|  
$$
where $c_1 $ is an external force and $\kappa [x,u]$ is  the Peach-Koehler force applied to
the curve ($N=2$ in this application). 

We briefly mentioned above that the Peach-Koehler force is 
created by the curve. Let us be a bit more specific. This force 
is computed through the resolution of an elliptic equation on a half space 
(corresponding to the law of linear elasticity). This equation is supplemented with
Dirichlet boundary conditions. On one hand, 
the boundary datum equals the indicator function of the interior of the curve.
On the other hand, loosely speaking, the force on the curve equals the normal derivative of 
the solution of the elliptic equation.
Hence, the integral operator which defines the Peach-Koehler force 
is a Dirichet-to-Neumann operator associated with an elliptic equation. 
In particular, the operator is singular. 

In order to define solutions for small times, the authors of \cite{ahlm} consider
a physically relevant regularized problem and $\kappa [x,u]$ reduces to 
$$
\int_{ \{ z : u(z) \ge 0 \} } c_0 (z) dz 
$$
with $c_0 \in W^{1,1} (\R^N)$. 

\noindent \emph{The major technical difficulty of this paper
is that $c_0$ does not have a constant sign and consequently, 
solutions corresponding to ordered initial data are not ordered;
in other words, comparison principle does not hold true.} 
In particular, this is one of the reasons
why solutions are constructed for small times. If $c_1$ is assumed to be large enough,
Alvarez, Cardaliaguet and Monneau \cite{acm} managed to prove the existence and uniqueness 
for large times. 

The difficulty related to comparison principle is circumvented in \cite{imr} by assuming
that the negative part of $c_0$ is concentrated at the origin. 
 The Peach-Koehler force $\kappa [x,u]$ (in the case of a single dislocation line) is defined in \cite{imr} as 
\begin{equation}\label{fmcm:def}
\int \mathrm{sign} (u (x+z) - u (x)) c_0 (z) dz = 
\int_{\{z : u (x+z) \ge u (x)\}} c_0 (z) dz
- \int_{\{z : u (x+z) < u (x)\}} c_0 (z) dz
\end{equation}
where $\mathrm{sign}(r)$  equals $1$ if $r \ge 0$ and $-1$ if
$r <0$. After an approximation procedure, 
the problem can be reduced to the study of 
$$
\partial_t U = \bigg[ c_1 (x) +  \int (U(x+z) -U (x)) c_0 (z) dz \bigg] |DU| 
$$
where $c_0$ is smooth, non-negative and of finite mass. 
We used the letter $U$ instead of $u$
in order to emphasize the fact that a change of unknown function is needed in order to reduce 
the study of the original equation to the study of this new one. 
\medskip

\noindent \emph{A second important remark is that solving such 
non-local eikonal equations does not permit to construct properly a
geometric flow.} 
More precisely, if the initial front $\partial \Omega_0$ is described with two
different initial functions $u_0$ and $v_0$, it is not sure that the
zero level sets of the corresponding solutions $u$ and $v$ coincide. 
In other words, the invariance principle does not hold true.

Still assuming that the negative part of $c_0$ is
concentrated at the origin, a good geometric definition of the flow is obtained in \cite{fim}
by considering a formulation ``\`a la Slep\v{c}ev'' of the geometric flow. The equation now becomes
\begin{equation}\label{eq:dislo-nonsing}
\partial_t u = \bigg[ c_1 (x) +  \int_{\{ z : u(t,x+z) > u(t,x) \}} c_0 (z) dz \bigg] |Du| \, .
\end{equation}
\noindent 
\emph{We point out that, with such a formulation, we cannot deal with singular  potentials $c_0$. }

Notice that in \cite{fim}, several fronts move, and they are interacting. The motion of a single front is a special case. 
Eventually, existence results of very weak solutions in a very general setting are obtained in \cite{bclm} and 
uniqueness is studied in \cite{bclm2}. 
\medskip

In \cite{dfm}, it is proved that if $c_0 (z)$ is smooth and regular near the origin
and behaves exactly like $|z|^{-N-1}$ at infinity, then a proper rescaling of \eqref{eq:dislo-nonsing}
converges towards the mean curvature motion (see Proposition~\ref{prop:dfm} and Theorem~\ref{thm:dfm} 
below). 

We finally mention that Caffarelli and Souganidis \cite{cafsoug} 
consider a Bence-Merriman-Osher scheme with kernels associated with the
fractional heat equation (that is to say the heat equation where the
usual Laplacian is replaced with the fractional one). They prove that
this scheme approximates the geometric flow at stake in this paper.

\subsubsection*{Phasefield theory for fractional reaction-diffusion equations}

Our second main motivation comes from phasefield theory for fractional reaction-diffusion 
equations \cite{is}. If one considers for instance stochastic Ising models with Ka\u{c} potentials with
very slow  decay at infinity (like a power law with proper exponent), then the study of the 
resulting mean field equation (after proper rescaling) is closely related to phasefield theory
for fractional reaction-diffusion equations such as
$$
\partial_t u^\eps + (-\Delta)^{\alpha/2} u^\eps +   \frac1{\eps^{1+\alpha}} f (u^\eps)  =0
$$
where $ (-\Delta)^{\alpha/2} $ denotes 
the fractional Laplacian with $\alpha \in (0,1)$ (in the case presented
here) and $f$ is a bistable  non-linearity. 
In particular, it is \emph{essential} in the analysis to deal with singular potentials.
Indeed, we have to be able to treat the case where 
$$
c_0 (z) = \frac{1}{|z|^{N+\alpha}}
$$
with $\alpha \in (0,1)$. 
It is also convenient to use the notion of generalized flows introduced by Barles and Souganidis 
\cite{barlessouganidis}
in order to develop a phasefield theory for such reaction-diffusion equations. 
See \cite{is} for further details and \cite{fm} for analogous problems.

\subsection{A new formulation}

The main contributions of this paper are the following:
\begin{itemize}
\item to give a proper level set formulation of dislocation dynamics for
  singular interaction potentials; in particular, sufficient conditions on the singularity
  to get stability results and comparison principles  are exhibited;
\item to shead light on the fact that the integral operator measures in
  a non-local way the curvature of the interface;
\item to study the geometric flow in detail: consistency of the
  definition, equivalent definition in terms of generalized flows,
  motion of bounded sets \textit{etc}. 
\end{itemize}

Because $\nu (dz) = c_0(z)dz$ is singular, we cannot define $\kappa [x,u]$ as in \eqref{eq:dislo-nonsing}.
Indeed, we must \emph{compensate} the singularity as it is commonly done
in order to get a proper integral representation of the fractional
Laplacian. We recall that the fractional Laplacian can be defined as follows
$$
(- \Delta)^{\alpha/2} u  (x) = - c_N (\alpha) \, \int (u(x+z) - u (x)) \frac{dz}{|z|^{N+\alpha}}
$$
where $c_N (\alpha)$ is a given positive constant depending on $N, \alpha$. 
Notice that if $\alpha < 1$ and $u$ is Lipschitz continuous at $x$ and $u$ is globally bounded, 
the integral is well defined. If $\alpha \ge 1$, the integral is not convergent
in the neighbourhood of $z=0$. In this case, the fractional Laplacian is defined either by considering
the principal value of the previous singular integral or by writing
$$
(- \Delta)^{\alpha/2} u  (x) = - c_N (\alpha) \, \int (u(x+z) - u (x) - Du (x) \cdot z \un_B (z)) 
\frac{dz}{|z|^{N+\alpha}}
$$  
where $\un_B (z)$ denotes the indicator function of the unit ball $B$. Notice that we used
the fact that the singular measure 
\begin{equation}\label{ex:se}
\nu (dz) = \frac{dz}{|z|^{N+\alpha}} 
\end{equation}
(with $\alpha \in (0,2)$) is even in order to get (at least formally)
$$
\int (Du (x) \cdot z \un_B(z)) \frac{dz}{|z|^{N+\alpha}} = 0 \, .
$$ 

As far as the fractional mean curvature is concerned, we must compensate the singularity of the
measure $\nu$ in a geometrical way. We explain how to do it when $\nu (dz) = c_0(z) dz$ with
$c_0 (z) = |z|^{-N-\alpha}$. Hence, we start from \eqref{fmcm:def}.
We use the fact that $c_0$ is even in order to get (formally)
$$
\nu ( z \in \R^N : Du (x) \cdot z \ge 0 )=\nu ( z \in \R^N : Du (x) \cdot z < 0 ) \, .
$$
Straightforward computations yield 
\begin{multline*}
\int \mathrm{sign} (u (x+z) - u (x)) c_0 (z) dz =\\ \nu \{ z : u(x+z) \ge u(x), Du (x) \cdot z  \le 0 \} 
-  \nu \{ z : u(x+z) < u(x), Du (x) \cdot z > 0 \} \, .
\end{multline*}
We thus define an integral operator $\kappa[x,u]$
for a general singular non-negative measure $\nu$  as follows
\begin{equation}\label{formule1}
\kappa [x,u] = \nu \{ z : u(x+z) \ge u(x), Du (x) \cdot z  \le 0 \} 
-  \nu \{ z : u(x+z) < u(x), Du (x) \cdot z > 0 \} \, .
\end{equation}
We explain below in detail 
(see Lemma~\ref{lem:nonsing}) the rigourous links between the
different formulations we considered up to now. 

Notice that this definition makes sense even if $\nu$ is not even. We recall that
the fractional Laplacian is a L\'evy operator. Since  
 L\'evy operators \cite{applebaum} are defined for singular (L\'evy) measure that
are not necessarily even, this seems to be relevant to define the fractional mean curvature
for singular measures that are not necessarily even. 
\medskip

We can say that this singular integral operator measures in a non-local way
the curvature of the ``curve'' $\{ u = u(x)\}$. Indeed, loosely speaking, 
we can say that in Formula~\eqref{formule1}
the first part (resp. the second one) measures how concave (resp. convex) 
is the set $\Omega = \{ z : u(x+z) > u(x) \}$ 
``near $x$''.  Moreover, we prove (see Proposition~\ref{prop:dfm-improved} below) that,
when $\nu$ is given by \eqref{ex:se}, the function
$(1-\alpha) \kappa [x,u]$ converges as $\alpha \in (0,1)$ goes to $1$
towards the classical mean curvature of $\{u=u(x)\}$ at $x$.  
This is the reason why we refer to $\kappa[x,u]$ as the
fractional mean curvature of the curve $\{ u = u(x)\}$ at point $x$.

\subsubsection*{The variational case}

When the singular measure $\nu (dz)$ has the form 
$$
\nu (dz) = - \bigg( \nabla \cdot G (z) \bigg) dz
$$
for a vectorfield $G$, the previous singular integral operator can be written as follows
\begin{equation}\label{formule2}
\kappa [x,u] = \int_{\{z : u (x+z) = u (x) \}} \bigg( G(z) \cdot \frac{\nabla u (x+z)}
{|\nabla u (x+z)|} \bigg) \sigma (d z) 
- b_G \left(\frac{\nabla u (x)}{|\nabla u (x)|}\right) \left(\frac{\nabla u (x)}{|\nabla u (x)|}\right) \cdot \frac{\nabla u (x)}{|\nabla u (x)|} 
\end{equation}
where $\sigma$ denotes the surface measure on the ``curve'' $\{z : u (x+z) = u (x) \}$ and where
$b_G = \int_{\{z : \nabla u (x) \cdot z =0\}} G(z) \sigma (dz)$ is a vector field of $\R^N$. 

Remark that the example we gave above is of this form. Indeed
$$
\frac{dz}{|z|^{N+\alpha}}=   - \frac1\alpha \bigg( \nabla \cdot
\frac{z}{|z|^{N+\alpha}} \bigg) dz \, .
$$ 

It is quite clear on this new formula that the singular integral operator is
geometric (in the sense that it only depends on the curve and not and
its parametrization $u$) and ``fractional''. 

After this work was finished, we have been told that non-local minimal surfaces
are being studied by Caffarelli, Roquejoffre and Savin \cite{cafroqsav}. 
Loosely speaking, they study sets whose
indicator functions minimize a fractional Sobolev norm $\|\cdot \|_{H^\alpha}$, $\alpha \in (0,1)$.
They prove in particular that local minimizers are 
viscosity solutions of $\kappa [x,u] =0$.

\subsubsection*{Comments and related works}

We gave two different formulations in the case of singular potentials. 
 \emph{We think that Formulation~\eqref{formule1} is the proper one in order to
get a complete level set formulation of the geometric flow} even if
Formulation~\eqref{formule2} is somehow more intuitive since it 
only involves the curve itself. In particular, the approach proposed by
Slep\v{c}ev \cite{slepcev} can be adapted (see \eqref{slep} below). 
\medskip

The level set equation we study has the following form
\begin{equation}\label{eq:dislo}
\partial_t u = \mu (\widehat{Du}) 
\bigg[ c_1 (x) + \kappa [x,u] \bigg] |Du |  \quad \text{ in } (0,+\infty) \times \R^N 
\end{equation}
supplemented with the following initial condition
\begin{equation}\label{eq:ic}
u(0,x) = u_0 (x) \quad \text{ in } \R^N 
\end{equation}
where $\hat{p}$ denotes  $p / |p|$ if $p \neq 0$,
 $\mu$ denotes the mobility vector field, $c_1(x)$ is a driving force.

Equation~\eqref{eq:dislo} is a non-linear non-local Hamilton-Jacobi equation. A lot of papers
are dedicated to the study of such equations. In our case, the main technical 
issues are the definition of viscosity solutions, the proof of their
stability and the proof of a strong uniqueness result. We
somehow use  ideas from \cite{slepcev} and combine them with the ones from \cite{barlesimbert},
even if the results of these two papers do not apply to our equation. 
\medskip

From a physical point of view and as far as dislocation dynamics is concerned, the measure
$\nu (dz) = c_0(z)dz$ should be $\nu (dz) = g(z/|z|) |z|^{-N-1} dz$ but in this case, the fractional
mean curvature is not well defined (see Remark~\ref{alpha=1}). 
It is also physically relevant to say that close to the dislocation line, in the \emph{core} 
of the dislocation, 
the potential should be
regularized. On the other hand, it is important to assume that $\nu (dz) \sim  g(z/|z|) |z|^{-N-1} dz$ as
$|z| \to +\infty$
since this prescribes the long range interaction between dislocation lines.
Another way to understand this difficulty is to say that in the core of the dislocation, the potential
is very singular and the singularity should be compensated at a higher (second) order. On one hand, 
this can explain
the loss of inclusion principle for such flows (if one can define them for large times). On the other hand,
 one can think
that in this case, the first term in such an expansion should be a mean curvature term.
This can make sense since curvature terms are commonly used to describe dislocation dynamics.
It can be relevant to add one in \eqref{eq:dislo}. However, we choose not to do so in order to
avoid technicalities and keep clear some important points in the proof of the stability result and the comparison
principle. 
\medskip

In order to better understand properties of the fractional mean curvature flow, 
a deterministic zero-sum repeated game is constructed in \cite{ise} in the spirit of \cite{ks0,ks}.

\bigskip

\noindent \textbf{Organization of the article}.  
In Section~\ref{sec:prelim}, we first give the precise assumptions
we make on data. We next give the definition(s) of the fractional mean curvature $\kappa [x,\cdot]$. 
In Section~\ref{sec:visc}, we first give the definition of viscosity solutions for \eqref{eq:dislo}, 
we then state and prove stability results. We next obtain strong uniqueness results by establishing
comparison principles. We also construct solutions of \eqref{eq:dislo} by Perron's method. 
We finally give two convergence results which explain in which limit one recovers the classical mean
curvature equation. 
In Section~\ref{sec:lsa}, we verify that the zero levet set of the solution $u$ we constructed in the
previous section only depends on the zero level set of the initial condition. This provides a level set 
formulation of the geometric flow. In the last section, we give an alternative geometric definition of 
the flow in terms of generalized flows in the sense of \cite{barlessouganidis}. 
\medskip

\noindent \textbf{Notation.} $\S$ denotes the unit sphere of $\R^N$. The ball of radius $\delta$ centered
at $x$ is denoted by $B_\delta (x)$. If $x=0$, we simply write $B_\delta$ and if moreover $\delta =1$, we write
$B$. If $p \in \R^N \setminus \{ 0\}$, $\hat{p}$ denotes $p / |p|$. If  $A$ is a subset of $\R^d$ with $d=N,N+1$ 
for instance, 
then $A^c$ denotes $\R^d \setminus A$. For two subsets $A$ and $B$, $A \sqcup B$ denotes $A \cup B$ and means
that $A \cap B = \emptyset$. The function $\un_A (z)$ equals $1$ if $z \in A$ and $0$ if not. 
\medskip

\noindent \textbf{Acknowledgements.} This paper is partially supported by the ANR grant ``MICA''. 
The author thanks R. Monneau and P. E. Souganidis for the fruitful
discussions they had together. He also thanks G. Barles and the two referees for their
attentive reading of this paper before its publication. 

\section{Preliminaries}
\label{sec:prelim}

In this section, we make precise the assumptions we need on data and we give several definitions 
of the fractional mean curvature. 

\subsection{Assumptions}

Here are the assumptions we make on the singular measure throughout the
paper.
\bigskip

\noindent \textsc{Assumptions.}
\begin{itemize}
\item[(A1)]
The mobility function $\mu : \S \to (0,+\infty)$ is  continuous.
\item[(A2)] The driving force $c_1:\R^N \to \R$ is Lipschitz continuous.
\item[(A3)] The singular measure $\nu$ is a  non-negative Radon measure satisfying
\begin{equation}\label{cond:nu}
\left\{\begin{array}{ll}
\text{for all } \delta>0, &  \nu (\R^N \setminus B_\delta) < +\infty \, ,\\
\text{for all } r>0,e \in \S,  & \nu \{ z \in B : 
\, r |z \cdot e| \le  |z-(z\cdot e)e|^2\} < +\infty \, , \medskip\\
& \delta \nu (\R^N \setminus B_\delta) \to 0 \; \text{ as } \delta \to 0 \, , \\
\text{for all } e \in \S, & r \; \nu \{ z \in B : \, r |z \cdot e| \le  |z-(z\cdot e)e|^2\} \to 0 \; 
\text{ as } r \to 0  
\end{array}\right. 
\end{equation}
($B_\delta$ denotes the ball of radius $\delta$ centered at the origin and $B=B_1$) 
and the last limit is uniform with respect to unit vectors $e \in \S$. 
\item[(A4)]
The initial datum $u_0:\R^N \to \R$ is bounded and Lipschitz continuous. 
\end{itemize}
We point out that the set $\{ z \in B : \, r |z \cdot e| \le  |z-(z\cdot e)e|^2\}$ 
appearing in the second and  the fourth lines of \eqref{cond:nu} 
is the region between an upward and downward (with respect to vector $e$) parabola. 

Even if the assumptions on the singular measure look technical at first glance, they are quite
natural in the sense that they imply several important properties:
\begin{itemize}
\item the measure is bounded away from the origin;
\item the singularity at the origin (if any) is a weak singularity in 
the sense that the fractional mean curvature of regular curves can be defined;
if the reader thinks of the  example given in \eqref{ex:se}, 
this means that we choose $\alpha <1$;
\item Parabolas $\{ z : r z_N = |z'|^2 \}$ (which are the model regular curves for us) 
can be handled, even when they degenerate ($r\to 0$). 
\end{itemize}

\begin{ex} \label{stand-ex}
The Standing Example for the singular measure is 
$$
\nu_{SE} (dz) = g \left(\frac{z}{|z|} \right) \frac{dz}{|z|^{N + \alpha}} 
$$ 
with $g:\S \to (0,+\infty)$ continuous and $\alpha \in (0,1)$.
The measure in \eqref{ex:se} corresponds to the isotropic case ($g \equiv 1$). 
\end{ex}

\subsection{Fractional mean curvature}

In this subsection, we make precise the definition of fractional mean curvature. 
Our definition extends the ones given in  \cite{dfm,fim} where $\nu (dz) = c_0 (z) dz$
to the case of singular measures.  

Let us define the fractional curvature of a smooth curve $\Gamma = \{ x \in \R^N : u(x) =0 \}
= \partial \{ x \in \R^N : u (x) > 0 \}$ associated with $\nu$. 
If $u$ is $C^{1,1}$ and $Du (x) \neq 0$, then the following quantity is well defined 
(see Lemma~\ref{lem:elem} below)
\begin{equation}\label{kappa}
\begin{array}{rl}
\kappa^* [x,\Gamma] &= \kappa^* [x,u] = \kappa_+^* [x,u] - \kappa^-_* [x,u] \\
\kappa_* [x,\Gamma] &= \kappa_* [x,u] = \kappa^+_* [x,u] - \kappa_-^* [x,u] 
\end{array}
\end{equation}
where
\begin{equation}\label{def:kappapm}
\begin{array}{rl}
\kappa^+_* [x,u]  &=  \nu \big( z : u(x+z)> u(x) , \; D  u(x) \cdot z <0 \big) \\
\kappa^-_* [x,u]  &=  \nu \big( z: u(x+z) < u(x) ,\; Du (x) \cdot z > 0 \big) 
\end{array}
\end{equation}
and
\begin{eqnarray*}
\kappa_+^* [x,u] &= \nu \big (z : u(x+z)\ge u(x), \; Du (x) \cdot z \le 0 \big) \\
\kappa^*_- [x,u] &=\nu \big( z: u(x+z) \le u(x),\;  Du (x) \cdot z \ge 0 \big) \; .
\end{eqnarray*}
We will see later (see Lemma~\ref{lem:well_defined} below)
that these functions are semi-continuous and this explains
the choice of notation we made. In order to understand the way these quantities
are related to the geometry of the curve $\{ u = u(x) \}$, it is convenient to 
write for instance
$$
\kappa^+_* [x,u]  =  \nu \big( z : 0 < - Du (x) \cdot z < u(x+z) -u (x) - Du (x) \cdot z \big) \, .
$$
As shown on Figure~\ref{lafig}, $\kappa^+_* [x,u]$ measures how concave the curve is ``near'' $x$
and  $\kappa^-_* [x,u]$ how convex it is. 
\begin{figure}
\centerline{\input{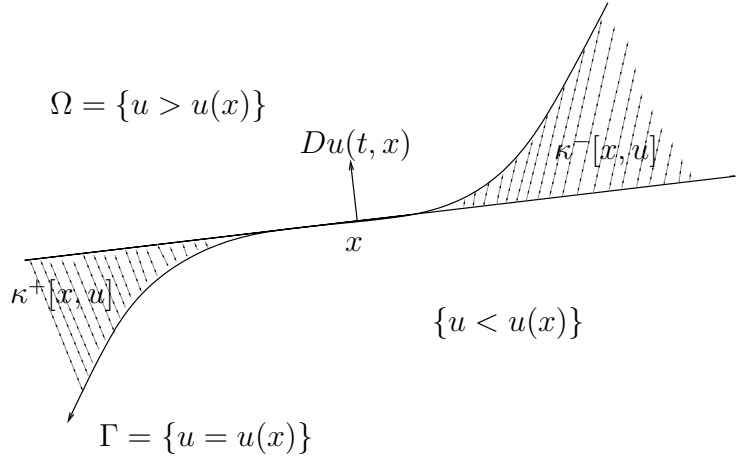}}
\caption{Fractional mean curvature of a curve}
\label{lafig}
\end{figure}
\begin{lem}[Fractional mean curvature is finite] \label{lem:elem}
If  $u$ is $C^{1,1}$ at point $x$, \textit{i.e.} there exists
a constant $C = C (x)> 0$ such that for all $z \in \R^N$
$$
|u(x+z) - u (x) - Du (x) \cdot z | \le C |z|^2
$$ 
and its gradient $D u (x) \neq 0$, 
then $\kappa_\pm^* [x,u]$   are finite. 

If $u$ is $C^{1,1}$ at $x$ and 
$D u \neq 0$ everywhere on $\{y \in \R^N :  u (y)= u(x) \}$ and $\nu$ is absolutely
continuous with respect to the Lebesgue measure, then $\kappa_\pm^* [x,u]$
are finite and
$$
\kappa^* [x,u] = \kappa_* [x,u] \, .
$$
\end{lem}
\begin{rem}
\label{alpha=1}
One can check that this lemma is false if $\alpha =1$ in the Standing Example~\ref{stand-ex}.
\end{rem}
\begin{proof}
We only prove the first part of the Lemma since the second part is clear. 

Since $\nu$ is bounded on $\R^N \setminus B_\delta$ for all $\delta >0$, 
it is enough to consider
\begin{eqnarray*}
(\kappa_+^*)^{1,\delta} [x,u] &=& \nu \big (z \in B_\delta: u(x+z)\ge u(x), \; Du (x) \cdot z \le 0 \big) \\
& =& \nu \big (z \in B_\delta: 0 \le r e \cdot z \le u (x+z) - u(x) + r e\cdot z \big) 
\end{eqnarray*}
where $r = |Du (x)|\neq 0$ and $e = r^{-1}Du (x)$. If now $z_N$ denotes $e \cdot z$ and $z' = z -z_N e$,
and if we choose $\delta$ such that $r - C\delta >0$,
we can write
\begin{eqnarray*}
(\kappa_+^*)^{1,\delta} [x,u] & \le & \nu \big (z \in B_\delta: 0 \le r z_N \le C z_N^2 + C |z'|^2)\\
& \le & \nu \big (z \in B_\delta: 0 \le C^{-1} (r-C\delta) z_N \le  |z'|^2)
\end{eqnarray*}
and the result now follows from Condition~\eqref{cond:nu}. 
\end{proof}

The following lemma explains rigourously the link between \eqref{eq:dislo} and \eqref{eq:dislo-nonsing}
and the link with the formulation used in \cite{fim} in the case where $\nu$ is a bounded measure. 
\begin{lem}[Link with regular dislocation dynamics]\label{lem:nonsing}
Consider $c_0 \in L^1 (\R^N)$ such that $c_0 (x) = c_0 (-x)$. Then
\begin{eqnarray*}
\int_{\{ z : u(t,x+z) > u(t,x) \}} c_0 (z) dz &=& \frac12 \int c_0 + \kappa_* [x,u] \\
\int \mathrm{sign}^* (u (x+z) - u (x)) c_0 (z) dz & =& \frac12  \kappa^* [x,u] \\
\int \mathrm{sign}_* (u(x+z) - u (x)) c_0 (z) dz & = & \frac12  \kappa_* [x,u] 
\end{eqnarray*}
with $\mathrm{sign}^* (r) = 1$ (resp. $\mathrm{sign}_* (r) = 1$) if $r
\ge 0$ (resp. $r >0$) and $-1$ if not and with $\nu (dz) = c_0 (z) dz$.
\end{lem}
Since the proof is elementary, we omit it. 

We conclude this section by stating two results which explain the link between two
special cases of fractional mean curvature operator and the classical mean curvature 
operator. The first one appears in \cite{dfm} (see their Corollary 4.2).
We state it in a special case in order to simplify the presentation. 
\begin{prop}[From dislocation dynamics to mean curvature flow -- \cite{dfm}]
\label{prop:dfm}
Assume that $\nu= \nu^\eps$ has the following form
$$
\nu (dz) = \nu^\eps ( dz ) = \frac1{\eps^{N+1}|\ln \eps |} c_0 \left( \frac{z}\eps \right) dz 
$$
with $c_0$ even, smooth, non-negative and such that $c_0 (z ) = |z|^{-N-1}$ for $|z| \ge 1$. 

Assume that $u \in  C^2( \R^N)$ and $Du (x) \neq 0$. Then
$$
\kappa [x,u] = \kappa^\eps [x,u ] \to C \, \mathrm{div} \, \left( \frac{Du}{|Du|} \right) (x) 
$$
as $\eps \to 0$ for some constant $C>0$.
\end{prop}
\begin{rem}
In \cite{dfm}, general anisotropic mean curvature operators can be obtained by considering
anisotropic measures $\nu (dz)$. 
\end{rem}
This result can be compared with the following one.
\begin{prop}[From fractional mean curvature to mean curvature]\label{prop:dfm-improved}
Assume that $\nu$ has the following form
$$
\nu (dz) = \nu^\alpha (dz) = (1-\alpha) \frac{dz}{|z|^{N+\alpha}}
$$
with $\alpha \in (0,1)$. 

Assume that $u \in C^2 (\R^N)$ and $Du (x) \neq 0$. Then
$$
\kappa [x,u] = \kappa^\alpha [x,u] \to C \, \mathrm{div} \, \left( \frac{Du}{|Du|}\right) 
$$
as $\alpha \to 1$, $\alpha <1$,  where $C$ is some positive constant.
\end{prop}
\begin{rem}
Anisotropic mean curvature can be obtained by considering 
$$
\nu^\alpha (dz ) = (1-\alpha) g \left(\frac{z}{|z|}\right) \frac{dz}{|z|^{N+\alpha}} \, .
$$
\end{rem}
\begin{proof}[Sketch of the proof of Proposition~\ref{prop:dfm-improved}]
For all $\eta$, we first choose $\delta$ such that 
\begin{equation}\label{eq:estim}
|u(x+z) - u (x) - Du (x) \cdot z - \frac12 D^2 u (x) z \cdot z| \le \eta |z|^2 \, . 
\end{equation}
If $e$ denotes $- Du (x)$ and $W(z)$ denotes $u(x+z) - u (x) - Du (x) \cdot z$, we have
\begin{eqnarray*}
\kappa^\alpha [x,u] &=& \nu^\alpha \{z \in \R^N : 
0 \le e \cdot z \le W(z) \} \\ && 
- \nu^\alpha \{ z \in \R^N : W(z) \le e \cdot z \le 0 \} \\
& =& (1-\alpha) \int_{\{z \in B_\delta : 0 \le e \cdot z \le W (z) \}} \frac{dz}{|z|^{N+\alpha}} -
(1-\alpha) \int_{\{z \in B_\delta : W(z) \le e \cdot z \le 0 \}} \frac{dz}{|z|^{N+\alpha}} \\
&& + O (1-\alpha) 
\end{eqnarray*} 
since $|z|^{-N -\alpha}$ is a bounded measure in $B_\delta^c$. 

In view of \eqref{eq:estim}, it is enough to prove the result for $W(z) = B z \cdot z$
where $B$ is a symmetric $N \times N$ matrix. Hence we study the convergence of 
\begin{eqnarray*}
K^\alpha  &= &
(1-\alpha) \int_{\{z \in B_\delta : 0 \le e \cdot z \le Bz \cdot z \}} \frac{dz}{|z|^{N+\alpha}} -
(1-\alpha) \int_{\{z \in B_\delta : B z \cdot z \le e \cdot z \le 0 \}} \frac{dz}{|z|^{N+\alpha}} \, .
\end{eqnarray*}
We next use the following system of coordinates: $z_1 = \hat{e} \cdot z$ 
and $z=(z_1, z')$.
We now write 
$$
Bz \cdot z = b_1 z_1^2 + z_1 (b_1' \cdot z') + B' z' \cdot z' 
$$
for some $b_1 \in \R$, $b_1' \in \R^{N-1}$ and a $(N-1)\times (N-1)$ symmetric matrix $B'$.
We thus want to prove
$$
K^\alpha \to |e|^{-1} \mathrm{tr} B'
$$
as $\alpha \to 1$. We can assume without loss of generality that $|e|=1$. 
For $z \in B_\delta$, we have
\begin{eqnarray*}
 e\cdot z \le B z \cdot z &\Rightarrow &   z_1 \le (1 -C\delta)^{-1} B' z' \cdot z' \\
z_1 \ge (1 -C\delta)^{-1} B' z' \cdot z' &\Rightarrow &    e\cdot z \ge B z \cdot z \, .
\end{eqnarray*}
Hence, it is enough to study the convergence of 
$$
\tilde{K}^\alpha = (1-\alpha) \int_{\{(z_1,z') \in B_\delta : 0 \le z_1
\le B'z' \cdot z' \}} \frac{dz}{|z|^{N+\alpha}} -
(1-\alpha) \int_{\{(z_1,z') \in B_\delta : B' z' \cdot z' \le z_1 \le 0 \}} \frac{dz}{|z|^{N+\alpha}} \, .
$$
If $\sigma (d\theta)$ denotes the measure on the sphere $\mathbb{S}^{N-2}$, we can write
\begin{eqnarray*}
\tilde{K}^\alpha 
& = & (1-\alpha) \int_{\{(z_1,z') : |z'| \le \delta, 0 \le z_1
\le B'z' \cdot z' \}} \frac{dz}{|z|^{N+\alpha}} -
(1-\alpha) \int_{\{(z_1,z'): |z'|\le \delta, B' z' \cdot z' \le z_1 \le 0 \}} \frac{dz}{|z|^{N+\alpha}} \\
& = & (1-\alpha) \int_{\theta \in \mathbb{S}^{N-2}:B'\theta \cdot \theta \ge 0} 
\int_{r=0}^\delta \int_{z_1=0}^{r^2 B'\theta \cdot \theta} 
\frac{r^{N-2}}{(z_1^2 + r^2)^{(N+\alpha)/2}} \sigma (d\theta) dr dz_1 
\\
&& -(1-\alpha) \int_{\theta \in \mathbb{S}^{N-2}:B'\theta \cdot \theta \le 0} 
\int_{r=0}^\delta \int_{z_1 = r^2 B'\theta \cdot \theta}^0 
\frac{r^{N-2}}{(z_1^2 + r^2)^{(N+\alpha)/2}} \sigma (d\theta) dr dz_1  \, .
\end{eqnarray*}
We next make the change of variables $z_1 = r^2 \tau$ and we get
\begin{eqnarray*}
\tilde{K}^\alpha & =& \int_{\theta \in \mathbb{S}^{N-2}:B'\theta \cdot \theta \ge 0} 
(1-\alpha) \int_{r=0}^\delta r^{-\alpha} \int_{\tau=0}^{B'\theta \cdot \theta} 
\frac{1}{(r^2 \tau^2 + 1)^{(N+\alpha)/2}} \sigma (d\theta) dr d\tau - (\dots) \, .
\end{eqnarray*}
We finally remark that 
\begin{eqnarray*}
\forall r \in (0,\delta), \;  
\int_{\tau=0}^{B'\theta \cdot \theta} \frac{1}{(r^2 \tau^2 + 1)^{(N+\alpha)/2}}
\to B' \theta \cdot \theta \text{ as } \delta \to 0 \, .
\end{eqnarray*}
In particular, for $\delta$ small enough,
$$
(1-\eta) B' \theta \cdot \theta \le
 \int_{\tau=0}^{B'\theta \cdot \theta} \frac{1}{(r^2 \tau^2 + 1)^{(N+\alpha)/2}}
\le (1 +\eta) B' \theta \cdot \theta \, .
$$
It is now easy to conclude by remarking 
\begin{eqnarray*}
(1-\alpha) \int_0^\delta r^{-\alpha} = \delta^{1-\alpha} \\
\int_{\mathbb{S}^{N-2}} \theta \otimes \theta \sigma (d\theta) = C I_{N-1} 
\end{eqnarray*}
where $I$ denotes  the $(N-1) \times (N-1)$identity matrix and $C$ is a positive constant.
\end{proof}

\section{Viscosity solutions for \eqref{eq:dislo}}

\label{sec:visc}

\subsection{Definitions}

The viscosity solution theory introduced in \cite{slepcev}
suggests that the good notion of solution for the fractional equation~\eqref{eq:dislo}
is the following one. 
\begin{defi}[Viscosity solutions for \eqref{eq:dislo}]\label{def:visc-geom}
\begin{enumerate}
\item
An upper semi-continuous function $u : [0,T] \times \R^N$ 
is a \emph{viscosity subsolution} of \eqref{eq:dislo} if for every smooth test-function
$\phi$  such that $u-\phi$ admits a global zero maximum  at $(t,x)$, we have 
\begin{equation}\label{subsol}
 \partial_t \phi (t,x) \le \mu (\widehat{D\phi (t,x)}) 
\bigg[ c_1 (x) +  {\kappa}^*[x,\phi(t,\cdot)]  \bigg]
 |D  \phi|(t,x)\  
\end{equation}
 if  $D \phi (t,x) \neq 0$ and $\partial_t \phi (t,x) \le 0$ if not. 

\item
A lower semi-continuous function $u$ is a \emph{viscosity supersolution} of 
\eqref{eq:dislo} if for every smooth test-function $\phi$ 
such that $u-\phi$ admits a  global minimum $0$ at $(t,x)$, we have 
\begin{equation}\label{sursol} 
\partial_t \phi (t,x) \ge \mu (\widehat{D\phi (t,x)}) 
\bigg[ c_1 (x) + {\kappa}_*[x,\phi(t,\cdot)]  \bigg]
 |D  \phi|(x_0,t_0)  
\end{equation}
if $D\phi (t,x) \neq 0$ and $\partial_t \phi (t,x) \ge 0$ if not. 
\item
A locally bounded  function 
$u$ is a \emph{viscosity solution} of \eqref{eq:dislo} if $u^*$ (resp. $u_*$) is a 
subsolution (resp. supersolution).
\end{enumerate}
\end{defi}
\begin{rem}
\label{rem:strictmax}
Given $\delta>0$, the global extrema in Definition~\ref{def:visc-geom} can be assumed to be strict
in a ball of radius $\delta$ centered at $(t,x)$. Such a result is classically expected and the 
reader can have a look, for instance, at the proof of the stability result in \cite{barlesimbert}.
\end{rem}
If one uses the notation introduced in \cite{slepcev}, the equation reads 
\begin{equation}\label{slep}
\partial_t u + F (x,Du,\{ z : u(x+z) \ge u(x) \})=0
\end{equation}
with, for $x,p \in \R^N$ and $K \subset \R^N$,  
$$
F(x,p,K) = 
\left\{\begin{array}{ll}
-\mu (\hat{p} ) \bigg[ c_1(x) + \nu (K \cap \{ p \cdot z \le 0 \} )
- \nu ( K^c \cap \{ p \cdot z > 0 \} \bigg] |p| & \text{ if } p \neq 0 \, , \\
0 & \text{ if not} 
\end{array}\right. 
$$
where $K^c$ is the complementary set of $K$. With this notation in hand, one can check that 
this non-linearity does not satisfy Assumption~(F5) of \cite{slepcev}. The idea is to check
that, somehow, Assumption (NLT) in \cite{barlesimbert} is satisfied and stability results 
thus hold true. 

Let us be more precise. 
We previously associated with $\kappa[\cdot,\cdot]$ the following non-local operators
(see the proof of Lemma~\ref{lem:elem})
\begin{equation}\label{def:kappapm12}
\begin{array}{lll}
(\kappa^+_*)^{1,\delta} [x,\phi] 
&=&  \nu \big( z \in B_\delta : \phi(x+z)> \phi(x) , \; z \cdot  D  \phi(x) <0 \big)\, , \\
(\kappa^+_*)^{2,\delta} [x,p,\phi] 
&=& \nu \big( z \notin B_\delta : \phi(x+z)> \phi(x) , \; z \cdot p <0 \big) \,.
\end{array}
\end{equation}
In the same way, we can define 
\begin{itemize}
\item
the negative non-local curvature operators $(\kappa^-_*)^{i,\delta}$, $i=1,2$, 
\item
upper semi-continuous envelopes of these four integral operators $(\kappa_\pm^*)^{i,\delta}$, $i=1,2$,
\item
 and lower/upper semi-continuous total non-local curvature operators 
$(\kappa_*)^{i,\delta}$, $(\kappa^*)^{i,\delta}$, $i=1,2$. 
\end{itemize}
By using the idea of Lemma~\ref{lem:nonsing}, it is easy to see that
\begin{equation}\label{eq:astuce}
\left\{
\begin{array}{lll}
(\kappa^*)^{2,\delta} [x,p, u] &=& \nu \big(z \notin B_\delta : u (t,x+z) \ge u (t,x) \big) 
-  \nu (z \notin B_\delta : p \cdot z >0 ) \, ,\\
 (\kappa_*)^{2,\delta} [x,p, u] &=& \nu \big(z \notin B_\delta : u (t,x+z) > u (t,x) \big)
-  \nu (z \notin B_\delta : p \cdot z \ge 0 ) \, .
\end{array}
\right. 
\end{equation}
We can now state an equivalent definition of viscosity solutions of
\eqref{eq:dislo}.
\begin{defi}[Equivalent definition]\label{def:visc-geom-bis}
\begin{enumerate}
\item
An upper semi-continuous function $u : [0,T] \times \R^N$ 
is a viscosity subsolution of \eqref{eq:dislo} if for every smooth test-function
$\phi$  such that $u-\phi$ admits a maximum $0$ at $(t,x)$ on $B_\delta(t,x)$, we have 
\begin{equation}\label{subsol-equiv}
 \partial_t \phi (t,x) \le \mu (\widehat{D\phi (t,x)}) 
\bigg[ c_1 (x) +  (\kappa^*)^{1,\delta}[x,\phi(t,\cdot)] + (\kappa^*)^{2,\delta} [ x, u (t,\cdot) ]
\bigg]  |D  \phi|(t,x)\  
\end{equation}
 if  $D \phi (t,x) \neq 0$ and $\partial_t \phi (t,x) \le 0$ if not. 

\item
A lower semi-continuous function $u$ is a viscosity supersolution of 
\eqref{eq:dislo} if for every smooth test-function $\phi$ 
such that $u-\phi$ admits a  global minimum $0$ at $(t,x)$, we have 
\begin{equation}\label{sursol-equiv} 
\partial_t \phi (t,x) \ge \mu (\widehat{D\phi (t,x)}) 
\bigg[ c_1 (x) + (\kappa_*)^{1,\delta}[x,\phi(t,\cdot)] 
+ (\kappa_*)^{2,\delta} [ x, u (t,\cdot)]
\bigg] |D  \phi|(x_0,t_0)  
\end{equation}
if $D\phi (t,x) \neq 0$ and $\partial_t \phi (t,x) \ge 0$ if not. 
\item
A continuous function 
$u$ is a viscosity solution of \eqref{eq:dislo} if it is both a sub and super-solution.
\end{enumerate}
\end{defi}
\begin{rem}
Equivalent definitions of this type first appeared in \cite{sayah} and
since the proof is the same, we omit it. 
\end{rem}
\begin{rem}
Remark~\ref{rem:strictmax} applies to the equivalent definition too. 
\end{rem}
\begin{rem}
Definition~\ref{def:visc-geom-bis} seems to depend on $\delta$. But since all these definitions
are equivalent to Definition~\ref{def:visc-geom}, it does not depend on it. Hence, when proving
that a function is a solution of \eqref{eq:dislo}, it is enough to do it for a fixed (or not) 
$\delta >0$. 
\end{rem}

\subsection{Stability results}

\begin{thm}[Discontinuous stability] \label{thm:stab}
Assume (A1)-(A3). 
\begin{itemize}
\item
Let $(u_n)_{n \ge 1}$ be a family of subsolutions of  \eqref{eq:dislo} that
is locally bounded, uniformly with respect to $n$. Then its relaxed 
upper limit $u^*$ is a subsolution of \eqref{eq:dislo}. 
\item
If moreover, $u_n (0,x) = u_0^n (x)$, then for all $x\in \R^N$
$$
u^* (0,x) \le u_0^* (x)
$$
where $u_0^*$ is the relaxed upper limit of $u_0^n$. 
\item
Let $(u_\alpha)_{\alpha \in \mathcal{A}}$ be a family of subsolutions of \eqref{eq:dislo} that
is locally bounded, uniformly with respect to $\alpha \in \mathcal{A}$.
Then $\bar{u}$, 
the upper semicontinuous envelope of $\sup_\alpha u_\alpha$ is a subsolution of \eqref{eq:dislo}. 
\end{itemize}
\end{thm}
Even if this result follows from ideas introduced in \cite{barlesimbert} together with classical ones, we give
a detailed proof for the sake of completeness. 
\begin{proof}
We only prove the first part of the theorem since it is easy to adapt it to get a proof of the third part. 
The second one is very classical and can be adapted from \cite{alibaudimbert} for instance. 

Consider a test function $\varphi$ such that $u^* - \varphi$ attains a global maximum at $(t,x)$. 
We can assume (see Remark~\ref{rem:strictmax}) that $u^* - \varphi$ attains a strict maximum at $(t,x)$ 
on $B_\delta (t,x)$. Consider a subsequence $p = p(n)$ and $(t_p,x_p)$ such that 
$$
u^*(t,x) = \lim_{n \to +\infty} u_{p(n)} (t_p,x_p) \, . 
$$
Classical arguments show that $u_p - \varphi$ attains a maximum on $B_\delta (t,x)$ at $(s_p,y_p) 
\in B_\delta(t,x)$
and that
$$
(s_p,y_p) \to (t,x) \quad \text{and} \quad  u_p (s_p,y_p) \to u^* (t,x) \, .
$$ 
Since $u_p$ is a subsolution of \eqref{eq:dislo}, we have
\begin{multline*}
\partial_t \varphi (s_p,y_p) \le \\\mu (\widehat{D\varphi (s_p,y_p)}) 
\bigg[ c_1 (y_p) +  (\kappa^*)^{1,\delta}[y_p,\varphi(s_p,\cdot)] 
+ (\kappa^*)^{2,\delta} [ y_p,D_x \varphi (s_p,y_p), u (s_p,\cdot) ]
\bigg]  |D  \varphi|(s_p,y_p)\  
\end{multline*}
 if  $D \varphi (t_p,x_p) \neq 0$ and $\partial_t \varphi (t_p,x_p) \le 0$ if not. 
If there exists a subsequence $q$ of $p$ such that $D \varphi (s_q,y_q) = 0$, then it is easy
to conclude. We thus now assume that $D \varphi(s_p,y_p) \neq 0$ for $p$ large enough. 
In view of the continuity of $\mu$ and $c_1$, the following technical lemma 
permits to conclude. 
\end{proof}
\begin{lem} \label{lem:well_defined}
Assume that $D\varphi (s_p,y_p) \neq 0$ for $p$ large enough. 
\begin{itemize}
\item Assume moreover that $D \varphi (t,x) \neq 0$. 
Then 
\begin{eqnarray*}
(s,y) \mapsto (\kappa^*)^{1,\delta} [y, \varphi(s,\cdot)] \quad \text{ and }\quad 
(s,y) \mapsto (\kappa^*)^{2,\delta} [y,D_x \varphi (s,y), u_p (s,\cdot)]
\end{eqnarray*}
 are well defined for $i=1,2$
in a neighbourhood of $(t,x)$ and 
\begin{eqnarray*}
\limsup_p \bigg\{ (\kappa^*)^{1,\delta} [y_p, \varphi(s_p,\cdot)] \bigg\}
\le (\kappa^*)^{1,\delta} [x, \varphi(t,\cdot)] \\
\limsup_p \bigg\{ (\kappa^*)^{2,\delta} [y_p,,D_x \varphi (s_p,y_p), u(s_p,\cdot)] \bigg\}
\le (\kappa^*)^{2,\delta} [x, D_x \varphi(t,x),\varphi(t,\cdot)] 
\end{eqnarray*}
as soon as $u_p (s_p,y_p) \to u (t,x)$ as $p \to + \infty$. 
\item
Assume now that $D \varphi (t,x) =0$. Then, for $i=1,2$,
$$
\bigg[ (\kappa^*)^{1,\delta} [y_p,\varphi(s_p,\cdot)] 
+ (\kappa^*)^{2,\delta} [y_p,D\varphi(s_p,y_p),u(s_p,\cdot)] 
\bigg]  |D  \varphi|(s_p,y_p) \to 0 
\quad \text{ as } p \to +\infty \, .
$$
\end{itemize}
\end{lem}
As we shall see, this lemma is a consequence of the following one.
\begin{lem}[\cite{slepcev}]\label{lem:slepcev}
Consider $f_p$ and $g_p$ two sequences of measurable functions on a set $U$ 
and $f \ge \limsup^* f_p$, $g \ge \limsup^* g_p$, and 
$a_p,b_p$ two sequences of real numbers converging to $0$. Then
$$
\nu (\{ f_p \ge a_p, g_p \ge b_p\} \setminus \{ f \ge 0,g \ge 0 \}) \to 0 \quad \text{ as } n \to + \infty \, . 
$$
\end{lem}
We mention that in \cite{slepcev}, the measure is not singular and there is only one sequence
of measurable functions but the reader can check that the slightly more general version we gave here
can be proven with exactly the same arguments. An immediate consequence of the lemma is the
following inequality
$$
\limsup_p \nu (\{ f_p \ge a_p, g_p \ge b_p\}) \le \nu (\{ f \ge 0,g \ge 0 \}) \, .
$$
\begin{proof}[Proof of Lemma~\ref{lem:well_defined}]
Let us first assume that $D \varphi (t,x) \neq 0$. In this case, for $(s,y)$ close to $(t,x)$, 
$D \varphi (s,y) \neq 0$ and all the integral operators we consider here are well defined 
(see Lemma~\ref{lem:elem}). 
Recall next that, for $i=1,2$, $(\kappa^*)^{i,\delta} = (\kappa^*_+)^{i,\delta} - (\kappa_*^-)^{i,\delta}$.  
Hence, it is enough to prove that 
\begin{eqnarray*}
\limsup_p \bigg\{ (\kappa^*_+)^{1,\delta} [y_p, \varphi(s_p,\cdot)] \bigg\}
\le (\kappa^*_+)^{1,\delta} [x, \varphi(t,\cdot)] \, ,\\
\liminf_p \bigg\{ (\kappa_*^-)^{1,\delta} [y_p, \varphi(s_p,\cdot)] \bigg\}
\ge (\kappa_*^-)^{1,\delta} [x, \varphi(t,\cdot)] \, ,\\
\limsup_p \bigg\{ (\kappa^*_+)^{2,\delta} [y_p,D_x \varphi (s_p,y_p), u_p(s_p,\cdot)] \bigg\}
\le (\kappa^*_+)^{2,\delta} [x, D_x \varphi(t,x),u^*(t,\cdot)] \, ,\\
\liminf_p \bigg\{ (\kappa_*^-)^{2,\delta} [y_p,D_x \varphi (s_p,y_p), u_p(s_p,\cdot)] \bigg\}
\le (\kappa_*^-)^{2,\delta} [x, D_x \varphi(t,x),u^*(t,\cdot)] \, .
\end{eqnarray*}

In order to prove the first inequality above for instance, choose
$f_p (z) = \varphi (s_p,y_p + z)- \varphi (t,x)$, $a_p = \varphi (s_p,y_p)-\varphi(t,x)$, 
$g_p (z) = - D \varphi (s_p,y_p) \cdot z$, $b_p =0$ in Lemma~\ref{lem:slepcev}. 
\medskip

We now turn to the case $D \varphi (t,x) =0$. We look for $\delta = \delta_p$ that goes
to $0$ as $p \to +\infty$ such that $|D \varphi (s_p,y_p)| \le C \delta_p$ and 
$$
(\kappa^*_+)^{1,\delta_p}[y_p,\varphi(s_p,\cdot)] | D \varphi (s_p,y_p)| \to 0 
\quad \text{and} \quad
(\kappa_*^-)^{1,\delta_p}[y_p,\varphi(s_p,\cdot)] | D \varphi (s_p,y_p)| \to 0 
$$
as $p \to +\infty$. This is enough to conclude since Condition~\eqref{cond:nu} implies
that
\begin{eqnarray*}
(\kappa^*_+)^{2,\delta_p}[y_p,D\varphi(s_p,y_p), u(s_p,\cdot)] | D \varphi (s_p,y_p)| \to 0 \\
(\kappa_*^-)^{2,\delta_p}[y_p,D\varphi(s_p,y_p), u (s_p, \cdot)] | D \varphi (s_p,y_p)| \to 0 
\, .
\end{eqnarray*}
We only prove that the first limit equals zero since the argument is similar for the second one. 
If $r_p$ denotes $|D \varphi(s_p,y_p)|$ and $e_p$ denotes $-r_p^{-1}D \varphi (s_p,y_p)$, 
and $z_N = e_P \cdot z$ and $z' = z - z_N e_p$, then
 \begin{multline*}
(\kappa^*_+)^{1,\delta}[y_p,\varphi(s_p,\cdot)] | D \varphi (s_p,y_p)| \\
 =  r_p \nu ( z \in B_{\delta_p} : 0 \le r_p e_p \cdot z \le \varphi (s_p,y_p+z) -\varphi(s_p,y_p) +
r_p e_p \cdot z)  \\
 \le  r_p \nu ( z \in B_{\delta_p} : 0 \le r_p z_N \le C |z'|^2 + Cz_N^2 ) \\
   \le  r_p \nu ( z \in B_{\delta_p} : 0 \le r_p z_N \le C |z'|^2 + C\delta_p z_N ) 
\end{multline*}
where $C$ is a bound for second derivatives of $\varphi$ around $(t,x)$. Now if we choose
$\delta_p = r_p / (2C)$, we get
\begin{eqnarray*}
(\kappa^*_+)^{1,\delta}[y_p,\varphi(s_p,\cdot)] | D \varphi (s_p,y_p)| 
  & \le & r_p \nu ( z \in B_{\delta_p} : 0 \le (r_p/2C) z_N \le  |z'|^2)  \\
  & \le & r_p \nu ( z \in B : 0 \le (r_p/2C) z_N \le  |z'|^2)  
\end{eqnarray*}
and the last limit in \eqref{cond:nu} permits now to conclude. 
\end{proof}

\subsection{Existence and uniqueness results}

Let us first state a strong uniqueness result. 
\begin{thm}[Comparison principle]\label{thm:comp}
Assume (A1)-(A4). Assume moreover 
\begin{itemize}
\item[(A3')]
For all 
$e \in \S$ and $r \in (0,1)$
\begin{equation}\label{cond:nu2}
r \; \nu \{ z \in B_\delta : \, r |z \cdot e| \le  |z-(z\cdot e)e|^2\} \to 0 \quad \text{ as } \delta \to 0
\end{equation}
uniformly in $e$ and $r \in (0,1)$ and
\begin{equation}\label{cond:nu3}
\nu (dz) = J (z) dz \text{ with } J \in W^{1,1} (\R^N \setminus B_\delta) \text{ for all } \delta > 0 \, . 
\end{equation}
\end{itemize}
Consider a bounded and Lipschitz continuous function $u_0$. 
Let $u$ (resp. $v$) be a bounded subsolution (resp. bounded supersolution) of \eqref{eq:dislo}. 
If $u(0,x) \le u_0 (x) \le v (0,x)$, then $u \le v$ on $(0,+\infty) \times \R^N$. 
\end{thm}
The proof is quite classical. The main difficulty is to deal with the singularity
of the measure. 
\begin{proof}[Proof of Theorem~\ref{thm:comp}]
We classically consider $M = \sup_{t,x} \{ u (t,x) - v(t,x) \}$ and argue by contradiction by 
assuming $M >0$. We next consider the following approximation of $M$
$$
\tilde{M}_{\eps,\alpha} = \sup_{t,s >0 , x,y \in \R^N} \{ u (t,x) - v(s,y) - \frac{(t-s)^2}{2 \gamma} 
-e^{Kt} \frac{|x-y|^2}{2 \eps} - \eta t - \alpha |x|^2 \} \, .
$$
Since $u$ and $v$ are bounded, this supremum is attained at a point $(\tilde{t},\tilde{s},\tilde{x},\tilde{y})$. 
We first observe that $\tilde{M}_{\eps,\alpha} \ge M/2 \ge 0$ for $\eta$ and $\alpha$ small enough.
Since $u$ and $v$ are bounded, this implies in particular 
\begin{equation}\label{eq:penalisation}
\eta \tilde{t} + e^{Kt}\frac{|\tilde{x} - \tilde{y}|^2}{2 \eps} + \alpha |\tilde{x}|^2 \le C_0
\end{equation}
where $C_0 = \| u\|_\infty + \|v\|_\infty$.

Classical results about penalization imply that $(\tilde{t},\tilde{s},\tilde{x},\tilde{y}) \to (\bar{t},\bar{t},
\bar{x},\bar{y})$ as $\gamma \to 0$ and $(\bar{t},\bar{t},\bar{x},\bar{y})$ realizes the following supremum
$$
M_{\eps,\alpha} = \sup_{t >0 , x,y \in \R} \{ u (t,x) - v(t,y) -e^{Kt}\frac{|x-y|^2}{2 \eps} - \eta t - \alpha |x|^2 \} \, .
$$
It is also classical \cite{cil92} to get, for $\eps$ and $\eta$ fixed,
\begin{equation}\label{eq:penal2}
\alpha |\bar{x}|^2 \to 0 \text{ as } \alpha \to 0 \, .
\end{equation}

We claim next that this supremum 
cannot be achieved at $t=0$ if $\eps,\alpha,\eta$ are small enough. 
To see this, remark first that
$M_{\eps,\alpha} \ge M/2 \ge 0$ for $\eta$ and $\alpha$ small enough 
and, if $\bar{t}=0$, use the fact that $u_0$ is Lipschitz continuous and get
$$
0 < \frac{M}2  
\le \sup_{x,y \in \R^N} \{ u_0 (x) - u_0 (y) - \frac{|x-y|^2}{2\eps} \} \le \sup_{r>0} \{ C_0 r -\frac{r^2}{2\eps} \}
=\frac12 C_0^2 \eps 
$$
and this is obviously false if $\eps$ is small enough. We conclude that, if the four parameters are small enough,
$\tilde{t}>0$ and $\tilde{s}>0$. 

Hence, we can write two viscosity inequalities. In order to clarify computations, we introduce the
function $M(p)$ defined as follows
$$
M (p) = \left\{ \begin{array}{ll} \mu (\hat{p}) |p|
& \text{ if } p \neq 0 \, ,\\0 & \text{ if } p=0 \, .\end{array} \right.
$$
It is easy to see that $M$ is uniformly continuous and it trivially satisfies
$$
|M(p)| \le \|\mu\|_\infty |p| \, .
$$
In the following, $\omega_M$ denotes the modulus of continuity of $M$.

We now write viscosity inequalities: for all $\delta >0$, 
\begin{eqnarray*}
\eta + \frac{\tilde{t}-\tilde{s}}\gamma + K e^{K\tilde{t}} \frac{|\tilde{x}-\tilde{y}|^2}{2\eps} 
&\le& \big( c_1(\tilde{x}) + 
 (\kappa^*)^{1,\delta} [\tilde{x}, \phi_u (\tilde{t},\cdot)] + (\kappa^*)^{2,\delta} [\tilde{x},\tilde{p} + 2 \alpha \tilde{x}, 
u (\tilde{t}, \cdot)] \big) M( \tilde{p} + 2 \alpha \tilde{x} ) \\
 \frac{\tilde{t}-\tilde{s}}\gamma &\ge& 
\big( c_1(\tilde{y}) + 
 (\kappa_*)^{1,\delta} [\tilde{y}, \phi_v (\tilde{s},\cdot)] + (\kappa_*)^{2,\delta} [\tilde{y},\tilde{p}, v(\tilde{s},\cdot) ] 
\big) M( \tilde{p} )
\end{eqnarray*}
where $\tilde{p} = e^{K \tilde{t}}\frac{\tilde{x}- \tilde{y}}\eps$ and 
\begin{eqnarray*}
\phi_ u (t,x) &=&v(\tilde{s},\tilde{y}) + 
\frac{(t-\tilde{s})^2}{2 \gamma} +  e^{Kt}\frac{|x -\tilde{y}|^2}{2 \eps} + \eta t + \alpha |x|^2 \, , 
\\
\phi_ v (s,y) &=& u(\tilde{t},\tilde{x}) 
-\frac{(s-\tilde{t})^2}{2 \gamma} - e^{K \tilde{t}}\frac{|y -\tilde{x}|^2}{2 \eps} - \eta \tilde{t} 
- \alpha |\tilde{x}|^2 \, . 
\end{eqnarray*}
Substracting these inequalities yield
\begin{equation}\label{eq:apasser}
\eta +  K e^{K\tilde{t}} \frac{|\tilde{x}-\tilde{y}|^2}{2\eps} 
\le  \|\mu\|_\infty \| Dc_1 \|_\infty  e^{K\tilde{t}} \frac{|\tilde{x}-\tilde{y}|^2}{\eps} 
  + \|c_1\|_\infty \omega_M (2 \sqrt{C_0 \alpha}) + T_{nl}
\end{equation}
(we used \eqref{eq:penalisation}) where
\begin{eqnarray*}
T_{nl} &=&  
\bigg((\kappa^*)^{1,\delta} [\tilde{x}, \phi_u (\tilde{t},\cdot)] 
+ (\kappa^*)^{2,\delta} [\tilde{x},\tilde{p} + 2 \alpha \tilde{x},u (\tilde{t}, \cdot)] 
\bigg) M( \tilde{p} + 2 \alpha \tilde{x} )  \\
&&
-\bigg((\kappa_*)^{1,\delta} [\tilde{y}, \phi_v (\tilde{s},\cdot)] 
+ (\kappa_*)^{2,\delta} [\tilde{y},\tilde{p},v(\tilde{s},\cdot) ] 
\bigg) M( \tilde{p} ) \, .
\end{eqnarray*}
Our task is now to find $\delta = \delta (\alpha,\eps)$ so that the right hand side of this 
inequality is small when the four parameters are small. We distinguish two cases. 

Assume first that there exists a sequence $\alpha_n \to 0$ and $\eps_n \to 0$ such that 
$$
\tilde{p} = \tilde{p}_n \to 0 \, .
$$
In this case, we simply choose $\delta =1$, $K= 2 \|\mu\|_\infty \|Dc_1\|_\infty$ and we pass to the limit as $n \to +\infty$ in 
\eqref{eq:apasser} and we get the desired contradiction: $\eta \le 0$. 

Assume now that for $\alpha$ and $\eps$ small enough, we have a constant $C_\eps$ independent
of $\alpha$ such that
\begin{equation}\label{secondcas}
| \tilde{p} | \ge C_\eps >0 \, .
\end{equation}
In this case, the following technical lemma holds true. 
\begin{lem}\label{lem:tech-comp}
By using \eqref{cond:nu2}, we have
$$
T_{nl} \le \frac1\eps o_\delta (1)  + \frac1\delta \omega_M( 2 \sqrt{C_0 \alpha}) + o_\alpha (1)[\eps] +   
C_\delta e^{K\tilde{t}} \frac{|\tilde{x}-\tilde{y}|^2}\eps
$$
where $C_0$ appears in \eqref{eq:penalisation}  and $C_\delta$ only depends on $J$, 
$\|\mu \|_\infty$ and 
$\delta$ (we emphasize that the third term goes to $0$ as $\alpha \to 0$ for fixed $\eps$).
\end{lem}
The proof of this lemma is postponed. We thus get (recall that $\tilde{p} = e^{K \tilde{t}}(\tilde{x}-\tilde{y})/\eps$) 
$$ 
\eta + K e^{K \tilde{t}} \frac{|\tilde{x}-\tilde{y}|^2}{2\eps}\le C  e^{K \tilde{t}} 
\frac{|\tilde{x}-\tilde{y}|^2}{\eps} + C(1 + \frac1\delta) \omega_M (2 \sqrt{ C_0 \alpha})  
+\frac1\eps  o_\delta (1) + o_\alpha (1) [\eps] + C_\delta e^{K \tilde{t}}\frac{|\tilde{x}-\tilde{y}|^2}\eps
$$
where  $C$ only depends on $c_1$, $\nu$ and $\|u\|_\infty + \|v\|_\infty$ and $C_\delta$ is given by the lemma.
By choosing $K  = 2 (C + C_\delta)$, we get
$$
\eta \le C (1 + \frac1\delta)\omega_M (2 \sqrt{C_0 \alpha}) + \frac1\eps  o_\delta (1) + o_\alpha (1) [\eps] \, .
$$ 
By letting successively $\alpha$ and $\delta$ go to $0$, 
we thus get a contradiction.
This achieves the proof of the comparison principle. 
\end{proof}

\begin{proof}[Proof of Lemma~\ref{lem:tech-comp}]
We first write
\begin{eqnarray*}
T_{nl} &\le&
\|\mu\|_\infty  | (\kappa^*)^{1,\delta}| [\tilde{x}, \phi_u (\tilde{t},\cdot)]  | \tilde{p} + 2 \alpha \tilde{x} | 
+ \|\mu \|_\infty |(\kappa_*)^{1,\delta}| [\tilde{y},\phi_v (\tilde{s},\cdot)]   |\tilde{p}|  \\
& & + |(\kappa^*)^{2,\delta}| [\tilde{x},\tilde{p} + 2 \alpha \tilde{x},u (\tilde{t},\cdot)] \omega_M (|2 \alpha \tilde{x}|)   \\
& & + \bigg( (\kappa^*)^{2,\delta} [\tilde{x},\tilde{p}+2\alpha \tilde{x} ,u (\tilde{t},\cdot)] - (\kappa_*)^{2,\delta} [\tilde{y},\tilde{p}, v (\tilde{s},\cdot)
]\bigg) M(\tilde{p}) \, .
\end{eqnarray*}
We thus estimate the right hand side of the previous inequality. 
We start with the first two integral terms.
\begin{eqnarray*}
| (\kappa^*)^{1,\delta}| [\tilde{x}, \phi_u (\tilde{t},\cdot)] & \le &
 (\kappa^*_+)^{1,\delta} [\tilde{x}, \phi_u (\tilde{t},\cdot)] +
 (\kappa^-_*)^{1,\delta} [\tilde{x}, \phi_u (\tilde{t},\cdot)] \\
 & \le & \nu (z \in B_\delta : 0 \le - (\tilde{p} + 2 \alpha \tilde{x} ) \cdot z \le 
(\alpha + e^{K \tilde{t}}/(2\eps)) |z|^2 ) \\
&& + \nu (z \in B_\delta : 0 > - (\tilde{p} + 2 \alpha \tilde{x} ) \cdot z> 
(\alpha + e^{K \tilde{t}}/(2\eps)) |z|^2 ) \\
& \le &  \nu (z \in B_\delta :   |\eps \tilde{p} + 2 \eps \alpha \tilde{x}| |e \cdot z| \le C(\eta) |z|^2 ) 
\end{eqnarray*}
where we use \eqref{eq:penalisation} to ensure, for $\alpha, \eps$ small enough,
$$
\frac{e^{K \tilde{t}}}{2} + \alpha \eps \le \frac12 e^{K C_0 /\eta} +1 := C(\eta) \, .
$$
If now $r_{\alpha,\eps}$ denotes $|\eps \tilde{p} + 2 \eps \alpha \tilde{x}|$ and we choose 
$\delta \le r_{\alpha,\eps} / (2 C(\eta))$, we use  \eqref{cond:nu2} to write
\begin{eqnarray*}
| (\kappa^*)^{1,\delta}| [\tilde{x}, \phi_u (\tilde{t},\cdot)] |  \tilde{p} + 2 \alpha \tilde{x} |
& =& \frac1\eps r_{\alpha,\eps} \nu (z \in B_\delta : r_{\alpha,\eps} |e \cdot z| \le C(\eta) |z|^2) \\
& \le & \frac1\eps r_{\alpha,\eps} \nu (z \in B_\delta : \frac12
r_{\alpha,\eps} |e \cdot z| \le  C(\eta) |z-(e\cdot z)e|^2) \\
&\le&  \frac{2 C(\eta)}\eps
\bigg\{ \sup_{e \in \S, r \in (0,1)} r 
\nu (z \in B_\delta :   r |e \cdot z| \le  |z - (e\cdot z) e|^2 ) \bigg\} \\
&=& \frac1\eps o_\delta (1) \, . 
\end{eqnarray*}
Since $\alpha \tilde{x} \to 0$ (see \eqref{eq:penalisation}), we choose for instance
$$
\delta \le \frac{\eps |\tilde{p}|}{4 C(\eta)} \, .
$$
Arguing similarly, we get for $\delta \le \frac{\eps |\tilde{p}|}{4 C(\eta)}$, 
$$
| (\kappa_*)^{1,\delta}| [\tilde{y}, \phi_v (\tilde{s},\cdot)] |\tilde{p}| \le  
 \frac1\eps o_\delta (1) \, . 
$$
As far as the third integral term is concerned, we simply write
$$
 |(\kappa^*)^{2,\delta}| [\tilde{x},u (\tilde{t},\cdot)] \omega_M(|2\alpha \tilde{x}|)
\le  \nu (B_\delta^c) \omega_M(|2\alpha \tilde{x}|) 
\le \frac1\delta \omega_M(2\sqrt{C_0 \alpha}) 
$$
(we used \eqref{eq:penalisation}).  We now turn to the last two
integral terms.  In view of \eqref{eq:astuce}, we can write
\begin{eqnarray*}
  \tilde{T}_{nl} &=&  
  (\kappa^*)^{2,\delta} [\tilde{x},\tilde{p} + 2 \alpha \tilde{x},u (\tilde{t},\cdot)] 
- (\kappa_*)^{2,\delta} [\tilde{y},\tilde{p},v (\tilde{s},\cdot)] \\
  &=& \nu (z \notin B_\delta : u(\tilde{t}, \tilde{x} +z ) \ge u (\tilde{t},\tilde{x}) ) 
  - \nu (z \notin B_\delta : v (\tilde{s},\tilde{y}+z) > v (\tilde{s}, \tilde{y}) ) \\
&& - \nu (z \notin B_\delta : (\tilde{p} + 2 \alpha \tilde{x}) \cdot z >0) 
+  \nu (z \notin B_\delta : \tilde{p}  \cdot z \ge 0) \,  .
\end{eqnarray*}
Now, we use \eqref{cond:nu3} to get
\begin{eqnarray*}
\tilde{T}_{nl} & = &  \int_{B_\delta^c} J (z-\tilde{x}) \un_{\{u (\tilde{t},\cdot) > u (\tilde{t}, \tilde{x})\}} (z ) dz 
-  \int_{B_\delta^c} J (z-\tilde{y}) \un_{\{v (\tilde{s},\cdot) > u (\tilde{s}, \tilde{y})\}} (z ) dz + o_\alpha (1)[\eps] \, .
\end{eqnarray*}
Remark next that the definition of $(\tilde{t},\tilde{s},\tilde{x},\tilde{y})$ implies the following
inequality: for all $z \in \R^N$,
$$
u (\tilde{t}, z ) - u (\tilde{t},\tilde{x}) \le v (\tilde{s},z ) 
- v (\tilde{s},\tilde{y}) + \alpha (|z|^2 - |\tilde{x}|^2) - e^{K \tilde{t}}\frac{|\tilde{x}-\tilde{y}|^2}{2\eps} \, .
$$
This implies that for $|z| \le R_{\alpha,\eps}$, we have
$$
\un_{\{u (\tilde{t},\cdot) > u (\tilde{t}, \tilde{x})\}} (z ) 
\le \un_{\{v (\tilde{s},\cdot) > u (\tilde{s}, \tilde{y})\}} (z )
$$
where 
\begin{eqnarray*}
R_{\alpha,\eps}^2 &=& \frac{1}\alpha \left( \alpha |\tilde{x}|^2 
+ e^{K \tilde{t}} \frac{|\tilde{x}-\tilde{y}|^2}{2\eps}\right) \\
& =& \frac1\alpha \left (o_\alpha (1) + \frac{\eps C_\eps^2}{4 C (\eta)} \right) \\
&\ge& \frac{\eps C_\eps^2}{8 C(\eta) \alpha} 
\end{eqnarray*}
where $C_\eps$ appears in \eqref{secondcas}. We used here \eqref{eq:penal2}. 
Hence, we have
\begin{eqnarray*}
\tilde{T}_{nl} & \le  & \int_{ |z| \ge R_{\alpha,\eps} } J (z - \tilde{x}) dz 
+ \int_{ z \in B_\delta^c } |J (z-\tilde{x}) - J (z -\tilde{y})| dz  \\
& \le & \int_{|\tilde{z}| \ge \frac{\sqrt{\eps} C_\eps}{2 \sqrt{8 C(\eta) \alpha}}}J(\tilde{z}) d\tilde{z}
 + C_\delta |\tilde{x}-\tilde{y}|  = 
o_\alpha (1)[\eps] + C_\delta |\tilde{x}-\tilde{y} |
\end{eqnarray*}
where we used once again \eqref{eq:penal2}. It is now easy to conclude.
\end{proof}
We now turn to the existence result. 
\begin{thm}[Existence]\label{thm:existence}
Assume (A1)-(A4) and (A3'). 
There then exists a unique bounded uniformly continuous 
viscosity solution $u$ of \eqref{eq:dislo},  \eqref{eq:ic}.  
\end{thm}
\begin{proof}
We first construct a solution for regular initial data. Precisely, we first assume that 
$u_0 \in C^2_b (\R^N)$ (the function and its first and second derivatives are bounded). 

Because we can apply Perron's method, it is enough to construct a sub- and a supersolution $u^\pm$ 
to \eqref{eq:dislo} such that $(u^+)_* (0,x) = (u^-)^* (0,x) = u_0 (x)$.
We assert that $u^\pm (t,x) = u_0 (x) \pm C t$ 
are respectively a super- and a subsolution of \eqref{eq:dislo} for $C$ large enough. 
To see this, we first prove that there exists $C_0 = C_0 (\| D^2 u_0 \|_\infty)$ such that
for all $x\in \R^N$ such that $Du_0 (x) \neq 0$, we have
\begin{equation}\label{estim-cle-existence}
( | \kappa_*| [x,u_0] + |\kappa^*| [x,u_0] ) |Du_0 (x) | \le C_0 \, .
\end{equation}
In order to prove this estimate, we simply write for $x$ such that $Du_0 (x) \neq 0$ 
\begin{eqnarray*}
( (\kappa_+^*)^{1,\delta} [x,u_0] + (\kappa_-^*)^{1,\delta} [x,u_0] )|Du_0 (x) |
\le 2 \nu (z \in B_\delta : r | e \cdot z | \le \frac12 \| D^2 u_0 \|_\infty |z|^2) r \\
\le 2 \nu (z \in B_\delta : r | e \cdot z | \le C |z-(e \cdot z) e|^2) r  \le C_\nu
\end{eqnarray*}
where $r = |Du_0 (x)|$, $C = \max (\|Du_0 \|_\infty,1)$ and $e = Du_0 (x) /r $ and $\delta = r / (2C)$ 
and $C_\nu$ is given by \eqref{cond:nu}. 
On the other hand
$$
( (\kappa_+^*)^{2,\delta} [x,u_0] + (\kappa_-^*)^{2,\delta} [x,u_0] )|Du_0 (x) | 
\le \frac{C_\nu}\delta r = 2 C_\nu C\, .
$$
We thus get Estimate~\eqref{estim-cle-existence}.

If now $u_0$ is not regular, we approximate it with $u_0^n \in C_b^2 (\R^N)$ and can prove that 
the corresponding sequence of solutions $u_n$ converges locally uniformly towards a solution $u$. 
Since this is very classical, we omit details (see for instance \cite{alibaudimbert}). 
\end{proof}

We now explain in which limit one recovers the mean curvature flow. To do so, we state two convergence
results. Their proofs rely on Propositions~\ref{prop:dfm} and \ref{prop:dfm-improved}. 
The first one (Theorem~\ref{thm:dfm}) 
appears in \cite{dfm} and the second one can be proved by using 
Proposition~\ref{prop:dfm-improved}. 
\begin{thm}[\cite{dfm}]\label{thm:dfm}
Assume that $\mu \equiv 1$, $c_1 \equiv 0$, $u_0$ is Lipschitz continuous and bounded and 
$$
\nu (dz) = \nu^\eps (dz) = \frac1{\eps^{N+1}|\ln \eps |} c_0 \left( \frac{z}\eps \right) dz 
$$
with $c_0$ even, smooth, non-negative and such that $c_0 (z ) = |z|^{-N-1}$ for $|z| \ge 1$.
Then the viscosity solution $u^\eps$ of \eqref{eq:dislo}, \eqref{eq:ic} converges locally 
uniformly as $\eps \to 0$ towards the viscosity solution $u$ of 
$$
\partial_t u = C |Du | \,  \mathrm{div} \left( \frac{Du}{|Du|} \right)
$$
($C$ is a positive constant) supplemented with the initial condition~\eqref{eq:ic}. 
\end{thm}
\begin{thm}\label{thm:dfm-improved}
Assume that $\mu \equiv 1$, $c_1 \equiv 0$, $u_0$ is Lipschitz continuous and bounded and 
$$
\nu (dz) = \nu^\alpha (dz) = (1-\alpha) \frac{dz}{|z|^{N+\alpha}}
$$
with $\alpha \in (0,1)$.
Then  the viscosity solution $u^\alpha$ of \eqref{eq:dislo}, \eqref{eq:ic} converges locally 
uniformly as $\alpha \to 1$ towards the viscosity solution $u$ of 
$$
\partial_t u = C |Du | \,  \mathrm{div} \left( \frac{Du}{|Du|} \right)
$$
($C$ is a positive constant) supplemented with the initial condition~\eqref{eq:ic}. 
\end{thm}

\section{The level set approach}
\label{sec:lsa}

In the previous section, we constructed a unique solution of 
\eqref{eq:dislo} in the case of singular measures satisfying (A3) and (A3')
and for bounded  Lipschitz continuous initial data (see (A4)). 
In the present section, we explain how to define a geometric flow 
by using these solutions of \eqref{eq:dislo}.  Precisely, we first prove (Theorem~\ref{thm:geom})
that if $u$ and $v$ are solutions of \eqref{eq:dislo} associated with two different initial
data $u_0$ and $v_0$ that have the same zero level sets, then so have $u$ and $v$. 
Hence, the geometric flows is obtained by considering the zero level sets of the solution $u$
of \eqref{eq:dislo} for any (Lipschitz continuous) initial datum. We also describe (Theorem~\ref{thm:extrem_sol}) 
the
 maximal and minimal discontinuous solutions of \eqref{eq:dislo} associated with an important class 
of discontinuous initial data.
\begin{thm}[Consistency of the definition]\label{thm:geom}
Assume (A1)-(A3) and (A3'). 
Let $u_0$ and $v_0$ be two bounded Lipschitz continuous functions 
and consider the viscosity solutions  $u$, $v$ associated with these initial conditions.
If 
\begin{eqnarray*}
\{ x \in \R^N : u_0 (x) > 0 \} &=& \{ x \in \R^N : v_0 (x) >0 \} \\
\{ x \in \R^N : u_0 (x) < 0 \} &=& \{ x \in \R^N : v_0 (x) <0 \} 
\end{eqnarray*}
then, for all time $t>0$, 
\begin{eqnarray*}
\{ x \in \R^N : u (t,x) > 0 \}& =& \{ x \in \R^N : v (t,x) >0 \} \\
\{ x \in \R^N : u (t,x) < 0 \} &=& \{ x \in \R^N : v (t,x) <0 \} 
\end{eqnarray*}
\end{thm}
In view of the techniques used to prove the consistency of the definition
of local geometric fronts (see for instance \cite{bss}), it is clear that this
 result is a straightforward consequence of the following proposition. 
\begin{prop}[Equation~\eqref{eq:dislo} is geometric]\label{prop:geom}
Consider $u:[0,+\infty) \times \R^N$ a bounded subsolution of \eqref{eq:dislo} and 
$\theta : \R \to \R$ a upper semi-continuous non-decreasing function. Then $\theta(u)$ 
is also a subsolution of \eqref{eq:dislo}. 
\end{prop}
Such a proposition is classical by now. It is 
proved by regularizing $\theta$ (in a proper way) 
with a strictly increasing function $\theta^n$,
by remarking that $\kappa^* [x, \theta^n (u)] = \kappa^*[x,u]$ in this case, 
and by using discontinuous stability. Details are left to the reader. 

Thanks to Theorem~\ref{thm:geom}, we can define a geometric flow in the following
way. Given $(\Gamma_0,D_0^+,D_0^-)$ such that $\Gamma_0$ is closed, $D^\pm_0$ are open 
and $\R^N = \Gamma_0 \sqcup D_0^+ \sqcup D_0^-$, we can write
$$
D_0^+  = \{ x \in \R^N : u_0 (x) >0 \},  D_0^- = \{ x \in \R^N : u_0 (x) <0 \},
\Gamma_0 = \{ x \in \R^N : u_0 (x) =0 \}
$$
for some bounded Lipschitz continuous function $u_0$ (for instance the signed distance function). 
If $u$ is the solution of \eqref{eq:dislo} submitted to the initial condition $u(0,x) = u_0 (x)$ 
for $x\in \R^N$, then Theorem~\ref{thm:geom} precisely says that the sets 
$$
D_t^+  = \{ x \in \R^N : u (t,x) >0 \}, D_t^- = \{ x \in \R^N : u(t,x) <0 \}, 
\Gamma_t = \{ x \in \R^N : u (t,x) =0 \}
$$
does not depend on the choice of $u_0$. 

The next theorem claims that there exists a maximal subsolution  minimal supersolution of \eqref{eq:dislo}
associated with the apropriate discontinuous initial data.
\begin{thm}[Maximal subsolution and minimal supersolution]\label{thm:extrem_sol}
Assume (A1)-(A3) and (A3'). Then the function 
 $\un_{D_t^+ \cup \Gamma_t} - \un_{D_t^-}$ (resp.  $\un_{D_t^+ } - \un_{D_t^- \cup \Gamma_t}$)
is the maximal subsolution (resp. minimal supersolution) of \eqref{eq:dislo} 
submitted to the initial datum  $\un_{D_0^+ \cup \Gamma_0} - \un_{D_0^-}$
(resp.  $\un_{D_0^+ } - \un_{D_0^- \cup \Gamma_0}$).
\end{thm}
This result is a consequence of Proposition~\ref{prop:geom} together with discontinuous stability
and the comparison principle. See \cite[p. 445]{bss} for details. 
\medskip

We conclude this section by showing that a bounded front propagates with finite speed. 
\begin{prop}[Evolution of bounded sets]\label{prop:bounded}
Assume (A1)-(A3) and (A3'). 
Let $\Omega_0$ be a bounded open set of $\R^N$: there exists $R>0$ such that $\Omega_0 \subset B_R$. 
Then the level set evolution $(\Gamma_t,D_t^+,D_t^-)$ of $(\partial \Omega_0,\Omega_0,(\bar{\Omega}_0)^c)$ 
 satisfies
$D_t^+ \cup \Gamma_t \subset \bar{B}_{R+Ct}$ with
$$
C = \|c_1 \|_\infty - \inf_{e \in \S} \nu (z \in \R^N : 0 \le e \cdot z \le |z|^2 )
$$
as long as $R +Ct >0$. 
\end{prop}
\begin{rem}
Another consequence of this proposition is that, if there are no driving force ($c_1 =0$), then
the set shrinks till it disappears. 
\end{rem}
\begin{proof}
The proof consists in constructing a supersolution of \eqref{eq:dislo}, \eqref{eq:ic}. 
It is easy to check that $C$ is chosen such that
$$
u(t,x) = Ct+ \sqrt{\eps^2 + R^2}  - \sqrt{\eps^2 + |x|^2}
$$ 
is a supersolution of \eqref{eq:dislo}. Since $\bar{B}_R = \{ x \in \R^N : u (0,x) \ge 0 \}$, we 
conclude that $D_t^+ \cup \Gamma_t \subset \{ x \in \R^N : u(t,x) \ge 0 \} =  \bar{B}_{R^\eps(t)}$ with 
 $R^\eps (t)= \sqrt{ (C t + \sqrt{\eps^2 + R^2})^2 - \eps^2}$. Hence, $D_t^+ \cup \Gamma_t \subset \cap_{\eps>0}
\bar{B}_{R^\eps (t)} = \bar{B}_{R+Ct}$. 
\end{proof}

\section{Generalized flows}

In this section, we follow \cite{barlessouganidis} and give an equivalent definition
of the flow by, freely speaking, replacing smooth test functions with smooth test fronts. 

In order to give this equivalent definition, we use the geometrical non-linearities we 
partially introduced in Section~\ref{sec:prelim}. For all $x,p \in \R^N$ and all
closed set $\mathcal{F} \subset \R^N$ and open set $\mathcal{O} \subset \R^N $
\begin{eqnarray*}
F_* (x, p, \mathcal{F} ) & =&  
\left\{\begin{array}{ll}
-\mu (\hat{p} ) \bigg[ c_1(x) + \nu (\mathcal{F} \cap \{ p \cdot z \le 0 \} )
- \nu ( \mathcal{F}^c \cap \{ p \cdot z > 0 \} \bigg] |p| & \text{if } p \neq 0 \, , \\
0 & \text{if not} \, , 
\end{array}\right. \\
F^* (x, p, \mathcal{O} ) & =&  
\left\{\begin{array}{ll}
-\mu (\hat{p} ) \bigg[ c_1(x) + \nu (\mathcal{O} \cap \{ p \cdot z < 0 \} )
- \nu ( \mathcal{O}^c \cap \{ p \cdot z \ge 0 \} \bigg] |p| & \text{ if } p \neq 0 \, , \\
0 & \text{ if not} \, .
\end{array}\right. \\
\end{eqnarray*}
We can now give the definition of a generalized flow. 
\begin{defi}[Generalized flows] \label{def:gen}
The family $(\mathcal{O}_t)_{t \in (0,T)}$ of open subsets of $\R^N$
(resp. $(\mathcal{F}_t)_{t \in (0,T)}$ of closed subsets of $\R^N$) 
is a \emph{generalized super-flow} (resp. \emph{sub-flow}) of \eqref{eq:dislo} if 
 for all $(t_0,x_0)  \in (0,+\infty) \times \R^N$, $r>0$, $h>0$,
and for all smooth function $\phi : (0;+\infty) \times \R^N \to \R$ such that 
\begin{enumerate}
\item \label{i} 
$\partial_t \phi  + F^* (x,D\phi, \{ z : \phi (t,x+z) > \phi (t,x) \} ) 
\le -\delta_\phi$ in $[t_0,t_0+h]\times \bar{B}(x_0,r)$ 

(resp.
$\partial_t \phi + F_* (x,D\phi, \{ z : \phi (t,x+z) \ge \phi (t,x) \}) \ge -\delta_\phi$ 
in $[t_0,t_0+h]\times \bar{B}(x_0,r)$)
\item \label{ii} $D\phi \neq 0$ in $\{ (s,y) \in [t_0,t_0+h] \times \bar{B}(x_0,r) : \phi(s,y)=0 \}$,
\item \label{iii} $\{ y \in \R^N : \phi (t_0,y) \ge 0 \} \subset \mathcal{O}^1_{t_0}$,

(resp. $\{ y \in \R^N : \phi (t_0,y) \le 0 \} \subset \R^N \setminus \mathcal{F}_{t_0}$),
\item \label{iv} 
$\{ y \notin \bar{B}(x_0,r) : \phi (s,y) \ge 0 \} \subset \mathcal{O}^1_s$ for all $s \in [t_0,t_0+h]$,

(resp. $\{ y \notin \bar{B}(x_0,r) : \phi (s,y) \le 0 \} \subset \R^N \setminus \mathcal{F}_s$ 
for all $s \in [t_0,t_0+h]$),
\end{enumerate}
then $ \{ y \in \bar{B}(x_0,r) :  \phi(t_0+h,y) > 0\} \subset \mathcal{O}^1_{t_0+h}$
(resp.  $ \{ y \in \bar{B}(x_0,r) :  \phi(t_0+h,y) < 0\} \subset \R^N \setminus \mathcal{F}_{t_0+h}$).
\end{defi}
Loosely speaking about generalized super-flows, 
Condition~\ref{i} says that in a prescribed neighbourhood $\mathcal{V}$ around $(t_0,x_0)$, 
the normal velocity of the test front $\{ \phi > 0 \}$ is strictly smaller than the one of the front $\mathcal{O}$;
Condition~\ref{ii} asserts that the front $\{ \phi =0 \}$ is smooth in $\mathcal{V}$; 
Conditions~\ref{iii} and \ref{iv} assert that the test front is inside the front $\mathcal{O}$ outside $\mathcal{V}$.
The conclusion is that the test front is inside the neighbourhood $\mathcal{O}$
at time $t+h$.
\begin{rem}
As far as local geometric fronts are concerned, Conditions~\ref{iii} and \ref{iv} impose that
the test front is inside $\mathcal{O}$ on the parabolic boundary of the neighbourhood. Here,
because the front is not local, the test front has to be inside $\mathcal{O}$ everywhere outside
the neighbourhood.
\end{rem}
\medskip

The next theorem asserts that Definition~\ref{def:gen} of the flow coincides 
with the level set formulation of Section~\ref{sec:lsa}. 
\begin{thm}[Generalized flows and level set approach]\label{thm:genflow}
Assume (A1)-(A3) and (A3'). 
Let  $(\mathcal{O}_t)_{t \in (0,T)}$ be a family of open subsets of $\R^N$
(resp. $(\mathcal{F}_t)_{t \in (0,T)}$ of closed subsets of $\R^N$) such that the set
$\cup_{t \in (0,T)} \{t\} \times \mathcal{O}_t $ is open in $[0,T] \times \R^N$ (resp. 
$\cup_{t \in (0,T)} \{t\} \times \mathcal{F}_t $ is closed in $[0,T] \times \R^N$).

Then $(\mathcal{O}_t)_{t \in (0,T)}$ (resp. $(\mathcal{F}_t)_{t \in (0,T)}$) 
is a generalized super-flow (resp. sub-flow) of \eqref{eq:dislo} if and only
if $\chi (t,x) = \un_{\mathcal{O}_t}(x) - \un_{\R^N \setminus \mathcal{O}_t}(x)$ 
(resp.  $\chi (t,x) = \un_{\mathcal{F}_t}(x) - \un_{\R^N \setminus \mathcal{F}_t}(x)$)
is a viscosity supersolution (resp. subsolution)
of \eqref{eq:dislo}, \eqref{eq:ic}.
\end{thm}
Since the proof of \cite{barlessouganidis} can be readily adapted, we omit it. 
We give a straightforward corollary of Theorems~\ref{thm:extrem_sol} 
and \ref{thm:genflow} that is used in \cite{is}. 
\begin{cor}[Abstract method] \label{cor:abstract}
Assume (A1)-(A3) and (A3'). 
Assume that $(\mathcal{O}_t)_t$ and $(\mathcal{F}_t)_t$ are respectively a generalized super-flow
and generalized sub-flow and suppose there exists two open sets $D_0^+,D_0^-$ such that $\R^N =
\partial \mathcal{O}_0 \sqcup D_0^+ \sqcup D_0^-$ and  such that $D_0^+ \subset \mathcal{O}_0$
and $D_0^- \subset \mathcal{F}_0^c$.
Then if $(\Gamma_t,D_t^+,D_t^-)$ denotes the level set evolution of $(\partial \mathcal{O}_0,D_0^+,D_0^-)$, 
we have for all time $t>0$
$$
D_t^+ \subset \mathcal{O}_t \subset D_t^+ \cup \Gamma_t, \quad D_t^- \subset \mathcal{F}_t^c \subset D_t^- \subset \Gamma_t \, .
$$
\end{cor}
\begin{rem}
One can check that under the assumptions of the previous corollary, we have in fact 
$D_0^+ = \mathcal{O}_0$ and 
$D_0^-  = \mathcal{F}_0^c$. 
\end{rem}
%

\bibliographystyle{siam}

\bibliography{fs}

\begin{thebibliography}{10}

\bibitem{alibaudimbert}
{\sc N.~Alibaud and C.~Imbert}, {\em Fractional semi-linear parabolic equations
  with unbounded data}, Trans. of the AMS, 361 (2009), pp.~2527--2566.

\bibitem{acm}
{\sc O.~Alvarez, P.~Cardaliaguet, and R.~Monneau}, {\em Existence and
  uniqueness for dislocation dynamics with nonnegative velocity}, Interfaces
  Free Bound., 7 (2005), pp.~415--434.

\bibitem{ahlm}
{\sc O.~Alvarez, P.~Hoch, Y.~Le~Bouar, and R.~Monneau}, {\em Dislocation
  dynamics: short-time existence and uniqueness of the solution}, Arch. Ration.
  Mech. Anal., 181 (2006), pp.~449--504.

\bibitem{applebaum}
{\sc D.~Applebaum}, {\em L\'evy processes and stochastic calculus}, vol.~93 of
  Cambridge Studies in Advanced Mathematics, Cambridge University Press,
  Cambridge, 2004.

\bibitem{bclm}
{\sc G.~Barles, P.~Cardaliaguet, O.~Ley, and R.~Monneau}, {\em Global existence
  results and uniqueness for dislocation equations}, SIAM J. Math. Anal., 40
  (2008), pp.~44--69.

\bibitem{bclm2}
{\sc G.~Barles, P.~Cardaliaguet, O.~Ley, and A.~Monteillet}, {\em Uniqueness
  results for nonlocal hamilton-jacobi equations}.
\newblock Submitted, 2008.

\bibitem{barlesimbert}
{\sc G.~Barles and C.~Imbert}, {\em Second-order elliptic integro-differential
  equations: Viscosity solutions' theory revisited}, Annales de l'Institut
  Henri Poincar\'e, Analyse non lin\'eaire, 25 (2008), pp.~567--585.

\bibitem{bss}
{\sc G.~Barles, H.~M. Soner, and P.~E. Souganidis}, {\em Front propagation and
  phase field theory}, SIAM J. Control Optim., 31 (1993), pp.~439--469.

\bibitem{barlessouganidis}
{\sc G.~Barles and P.~E. Souganidis}, {\em A new approach to front propagation
  problems: theory and applications}, Arch. Rational Mech. Anal., 141 (1998),
  pp.~237--296.

\bibitem{cafroqsav}
{\sc L.~Caffarelli, J.-M. Roquejoffre, and O.~Savin}, {\em Personal
  communication}, 2008.

\bibitem{cafsoug}
{\sc L.~Caffarelli and P.~E. Souganidis}, {\em Convergence of nonlocal
  threshold dynamics approximations to front propagation}, Arch. Rational Mech.
  Anal.,  (2008).
\newblock to appear.

\bibitem{cardaliaguet}
{\sc P.~Cardaliaguet}, {\em Front propagation problems with nonlocal terms.
  {II}}, J. Math. Anal. Appl., 260 (2001), pp.~572--601.

\bibitem{cgg}
{\sc Y.~G. Chen, Y.~Giga, and S.~Goto}, {\em Uniqueness and existence of
  viscosity solutions of generalized mean curvature flow equations}, J.
  Differential Geom., 33 (1991), pp.~749--786.

\bibitem{cil92}
{\sc M.~G. Crandall, H.~Ishii, and P.-L. Lions}, {\em User's guide to viscosity
  solutions of second order partial differential equations}, Bull. Amer. Math.
  Soc. (N.S.), 27 (1992), pp.~1--67.

\bibitem{dfm}
{\sc F.~Da~Lio, N.~Forcadel, and R.~Monneau}, {\em Convergence of a non-local
  eikonal equation to anisotropic mean curvature motion. {A}pplication to
  dislocations dynamics}, J. Eur. Math. Soc,  (2006).
\newblock To appear.

\bibitem{es}
{\sc L.~C. Evans and J.~Spruck}, {\em Motion of level sets by mean curvature.
  {I}}, J. Differential Geom., 33 (1991), pp.~635--681.

\bibitem{fim}
{\sc N.~Forcadel, C.~Imbert, and R.~Monneau}, {\em Homogenization of some
  particle systems with two-body interactions and of the dislocation dynamics},
  Discrete Contin. Dyn. Syst.,  (2008).
\newblock To appear.

\bibitem{fm}
{\sc N.~Forcadel and R.~Monneau}, {\em Existence of solutions for a model
  describing the dynamics of junctions between dislocations}, 2008.
\newblock Submitted.

\bibitem{imr}
{\sc C.~Imbert, R.~Monneau, and E.~Rouy}, {\em Homogénéization of first order
  equations with $u/\epsilon$-periodic hamiltonians. part {II}: application to
  dislocation dynamics}, Comm. in PDEs, 33 (2008), pp.~479 -- 516.

\bibitem{ise}
{\sc C.~Imbert and S.~Serfaty}, {\em Repeated games for eikonal equations,
  dislocation dynamics and parabolic integro-differential equations}, 2008.
\newblock In preparation.

\bibitem{is}
{\sc C.~Imbert and P.~E. Souganidis}, {\em Phasefield theory for fractional
  reaction-diffusion equations and applications}, 2008.
\newblock In preparation.

\bibitem{ks}
{\sc R.~Kohn and S.~Serfaty}, {\em A deterministic-control-based approach to
  fully non-linear parabolic and elliptic equations}, tech. rep., Courant
  Institute, New York University, 2007.

\bibitem{ks0}
{\sc R.~V. Kohn and S.~Serfaty}, {\em A deterministic-control-based approach to
  motion by curvature}, Comm. Pure Appl. Math., 59 (2006), pp.~344--407.

\bibitem{oshersethian}
{\sc S.~Osher and J.~A. Sethian}, {\em Fronts propagating with
  curvature-dependent speed: algorithms based on {H}amilton-{J}acobi
  formulations}, J. Comput. Phys., 79 (1988), pp.~12--49.

\bibitem{sayah}
{\sc A.~Sayah}, {\em \'{E}quations d'{H}amilton-{J}acobi du premier ordre avec
  termes int\'egro-diff\'erentiels. {I}. {U}nicit\'e des solutions de
  viscosit\'e}, Comm. Partial Differential Equations, 16 (1991),
  pp.~1057--1074.

\bibitem{slepcev}
{\sc D.~Slep{\v{c}}ev}, {\em Approximation schemes for propagation of fronts
  with nonlocal velocities and {N}eumann boundary conditions}, Nonlinear Anal.,
  52 (2003), pp.~79--115.

\end{thebibliography}

\end{document}